\documentclass[a4paper,11pt]{amsart}
\usepackage{amsmath, amsthm, amsfonts, amssymb}
\usepackage{mathtools}
\usepackage{enumerate}
\usepackage[bookmarks=false]{hyperref} 
\usepackage{color}

\title{The product of lattice covolume and discrete series formal dimension: $\fp$-adic $GL(2)$ }
\author{Lauren C. Ruth}

\address{Department of Mathematics, Vanderbilt University, 1326 Stevenson Center, Nashville, TN 37240, USA}
\email{lauren.c.ruth@vanderbilt.edu}

\newtheorem{thm}{Theorem}[section]
\newtheorem{prop}[thm]{Proposition}
\newtheorem{cor}[thm]{Corollary}
\newtheorem{lem}[thm]{Lemma}
\theoremstyle{definition}

\newtheorem{defn/lem}[thm]{Definition/Lemma}
\newtheorem{rem}[thm]{Remark}
\newtheorem{examp}[thm]{Example}

\newtheorem{problem}{Problem}[]

\newtheorem{que}[problem]{Question}

\newcommand{\C}{{\mathbb C}}
\newcommand{\F}{{\mathbb F}}

\renewcommand{\H}{{\mathbb H}}

\newcommand{\R}{{\mathbb R}}

\newcommand{\W}{{\mathbb W}}

\newcommand{\Z}{{\mathbb Z}}

\newcommand{\HH}{{\mathcal H}}
\newcommand{\JJ}{{\mathcal J}}

\newcommand{\MM}{{\mathcal M}}

\newcommand{\UU}{{\mathcal U}}

\newcommand{\fP}{\mathfrak{P}}
\newcommand{\fo}{\mathfrak{o}}
\newcommand{\fp}{\mathfrak{p}}

\newcommand{\id}{\operatorname{id}}

\begin{document}
\begin{abstract}
Let $F$ be a nonarchimedean local field of characteristic $0$ and residue field of order not divisible by $2$.  We show how to calculate the product of the covolume of a torsion-free lattice in $PGL(2,F)$ and the formal dimension of a discrete series representation of $GL(2,F)$.  The covolume comes from a theorem of Ihara, and the formal dimensions are contained in results of Corwin, Moy, and Sally.  By a theorem going back to Atiyah, and by triviality of the second cohomology group of a free group, the resulting product is the von Neumann dimension of a discrete series representation considered as a representation of a free group factor.  
\end{abstract}

\maketitle


\section{Introduction} 

The product of the covolume of a lattice and the formal dimension of a discrete series representation is significant in many areas of mathematics: Under various conditions, this product equals the multiplicity of a discrete series representations in square-integrable functions on a symmetric space, the dimension of a space of cusp forms on the upper-half plane, or the size of the commutant of a von Neumann algebra.  This note concerns the latter.

In Section \ref{background}, we provide background on von Nemann algebras; then we give examples of a theorem of Atiyah in the setting of $PSL(2,\R)$, observing that the list of examples can be extended if the lattice has trivial second cohomology group; then we explain how a proof of Atiyah's theorem carries over to the setting of $PGL(2,F)$, where $F$ is a local nonarchimedean field.  

In Section \ref{comp}, we deal carefully with Haar measure, obtaining the covolume of a lattice from a theorem of Ihara, and, for this same normalization of Haar measure, the formal dimensions of discrete series representations, which were calculated for $GL(n,F)$ in a paper of Corwin, Moy, and Sally. 

Our result is Example \ref{padicfree}:  

\begin{examp} \label{padicfree}
Let $F$ be a nonarchimedean local field of characteristic $0$ and residue field of order $q$ not divisible by $2$.  Let $F_n$ be a free group contained as a lattice in $PGL(2,F)$.
Let $(\pi,\HH)$ be a discrete series representation of $GL(2,F)$.  Then we have the following list of von Neumann dimensions for each $\HH$ as a representation of the II$_1$ factor $RF_n$: 
\begin{align*}
        \text{dim}_{R\Gamma} \HH &= \begin{cases} 
     n-1                      & \text{$\pi$ Steinberg representation (or twist)} \\ 
      (n-1)2q^k, \,\,\,\, k=0,1,2, \ldots  & \text{$\pi$ unramified supercuspidal representation}  \\
      (n-1)(q+1)q^{k}, \,\,\,\, k=0,1,2,\ldots   & \text{$\pi$ ramified supercuspidal representation}  \\
      \end{cases}
\end{align*}
\end{examp}

This extends the list of von Neumann dimensions in the setting of $PGL(2,F)$ in \cite{ruthdiss}, where only the Steinberg representation and just one supercuspidal representation were considered.

Compare this list of von Neumann dimensions for representations of free group factors $RF_n$ to the list in Example \ref{free} in the setting of $PSL(2,\R)$.  The difference between the two lists is due to the difference between the parametrizations of the discrete series representations of $SL(2,\R)$ and $GL(2,F)$.

\section{Background} \label{background}

\subsection{Von Neumann dimension}

A \textit{von Neumann algebra} is a self-adjoint unital subalgebra of bounded linear operators on complex Hilbert space that is closed in the weak operator topology.  A \textit{factor} is a von Neumann algebra whose center consists only of scalars.  A \textit{finite factor} is a factor possessing a (unique) faithful weakly continuous trace.  These are the factors of Type I$_n$, which are isomorphic to $M_n(\C)$, and of Type II$_1$, of which we will now give an example.  A discrete subgroup $\Gamma$ of a locally compact unimodular group $G$ is called a \textit{lattice} if $G / \Gamma$ supports a finite Haar measure; and by Borel's density theorem, any lattice in a centerless semisimple Lie group without compact simple factors has this property (\cite{ghj} Lemma 3.3.1).  Let $\Gamma$ be such a lattice, and denote by $R\Gamma$ the closure with respect to the weak operator topology of the right regular representation of $\Gamma$ on $l^2(\Gamma)$.  Then $R\Gamma$ is a II$_1$ factor (\cite{sak} Lemma 4.2.18).

While the normalized trace on a Type I$_n$ factor assumes values in $\lbrace 0, \frac{1}{n}, \frac{2}{n}, \ldots , 1 \rbrace$ on projections, 
the normalized trace on a Type II$_1$ factor assumes all values in $\left[ 0, 1 \right]$ on projections.  This leads to the continuous \textit{von Neumann dimension} of modules over II$_1$ factors.  When Murray and von Neumann defined this dimension in \cite{roo1}, they called it the ``coupling constant,'' because it measures the size of a factor's commutant on a given representation space.  Representations of a II$_1$ factor are classified up to unitary equivalence by their von Neumann dimension: Given a representation, one can obtain a representation of any von Neumann dimension in 
$\left( 0, \infty \right)$ by applying an amplification and a projection.  Atiyah used this notion of von Neumann dimension to define $L^2$-Betti numbers in \cite{ati}, and Theorem \ref{ghjvndim} below is an outgrowth of his work on index theory.  

\subsection{Formal dimensions of discrete series representations}

We call an irreducible unitary representation $(\pi,\HH)$ of a connected semisimple Lie group $G$ (\textit{e.g}.\ $SL(2,\R)$) a \textit{discrete series representation} if one (hence all) of its matrix coefficients are square-integrable. Such a representation is equivalent to a subrepresentation of the right regular representation of $G$ on $L^2(G)$ (\cite{rob}, 16.2 Theorem).  Discrete series representations behave like irreducible unitary representations of compact groups, in the sense that they have a \textit{formal dimension}, a constant d$_\pi > 0$ such that Schur's relations hold:  
\begin{align*}
    \int_G ( \pi(g) u, v ) \overline{( \pi(g) u', v' )} dg = \frac{1}{\text{d}_\pi}( u, u') \overline{(v, v')} \qquad (u,v,u',v' \in \HH)
\end{align*}

Consider $GL(2,F)$, where $F$ is any local field.  Similar to the definition above, we call an irreducible unitary representation $(\pi,\HH)$ of $GL(2,F)$ a \textit{discrete series representation} if one (hence all) of its matrix coefficients $( \pi(g) u,v) )$, $u,v \in \HH$, are square-integrable \textit{modulo the center $Z$ of $G$}.  In this setting, the formal dimension d$_\pi$ is again defined by Schur's relations:
\begin{align*}
    \int_{G/Z} ( \pi(g) u, v ) \overline{( \pi(g) u', v' )} d\dot{g} = \frac{1}{\text{d}_\pi}( u, u') \overline{(v, v')}  \qquad (u,v,u',v' \in \HH)
\end{align*}

\subsection{Motivating theorem and examples}

The following theorem has its roots in Atiyah's work on $L^2$-index in \cite{ati} and in Atiyah and Schmid's geometric realizations of discrete series representations of semisimple Lie groups in \cite{atishm}.

\begin{thm} \label{ghjvndim}
(\cite{ghj} Theorem 3.3.2) 
Let $G$ be a connected semisimple real Lie group having discrete series representations, let $\Gamma$ be a lattice in $G$, and let $\pi : G \rightarrow U(\HH)$ be a discrete series representation. Assume that every non-trivial conjugacy class of $\Gamma$ has infinitely many elements. Then the representation $\pi$ restricted to $\Gamma$ extends to a representation of the \emph{II}$_1$ factor $R\Gamma$, with von Neumann dimension given by 
\begin{align*}
    \emph{dim}_{R\Gamma} \HH = \emph{covol}(\Gamma) \cdot \emph{d}_\pi,
\end{align*}
where $\emph{d}_\pi$ denotes the formal dimension of $\pi$.
\end{thm}

Let us give examples.  Let $G=PSL(2,\R)$, and let $\Gamma$ be a Fuchsian group of the first kind.  By a theorem of Fricke and Klein (Proposition 2.4 in \cite{iwan}), $\Gamma$ is isomorphic to a group having  generators
\begin{align} \label{Fgens}
    A_1, \ldots , A_g, B_1 , \ldots B_g , E_1 , \ldots , E_l, P_1,...P_h
\end{align}
satisfying the relations 
\begin{gather} \label{Frelns}
\left[ A_1,B_1 \right] \ldots \left[ A_g,B_g \right] E_1 \ldots E_l P_1 \ldots P_h = 1, \\
E_j ^{m_j} = 1, \,\,\,\,\, 1 \leq j \leq l. \nonumber
\end{gather}

The area of a fundamental domain for the action of $\Gamma$ by linear fractional transformations on the upper half-plane $\H$ with respect to the measure $y^{-2}dxdy$ is given by the Gauss--Bonnet formula (pg. 33 of \cite{iwan}),
\begin{align} \label{gb}
\frac{\text{vol}( \Gamma \backslash \H )}{2 \pi} = 2g-2+\sum_{j=1}^l \left( 1-\frac{1}{m_j} \right) + h.
\end{align}

Let $N.A.K$ be the Iwasawa decomposition of $SL(2,\R)$, and normalize Haar measure $dn.da.dk$ so that $\int_K dk = 1$.  Then $\text{vol}(\Gamma \backslash \H) = \text{vol}(\Gamma \backslash G)$.

With respect to this same Haar measure, the list of formal dimensions of discrete series representations of $SL(2,\R)$ is 
\begin{align} \label{rformdim}
 \text{d}_k = \frac{k-1}{4\pi} \qquad (k=2,3,4,\ldots)
\end{align}
The calculation of formal dimensions is part of 17.8 Theorem in \cite{rob}, and the Haar measure normalization is discussed at the top of pg.\ 148 of \cite{ghj}.  

From the construction in Section 17 of \cite{rob}, one sees that the discrete series representations of $SL(2,\R)$ factoring through $PSL(2,\R)$ are those indexed by even $k$.  (For odd $k$, $-I$ acts as multiplication by $-1$.)  So we have the following example.  

\begin{examp}  \label{keven}
Let $G=PSL(2,\R)$, let $\Gamma$ be a Fuchsian group of the first kind in $G$, and let $\pi : G \rightarrow \UU(\HH)$ be a discrete series representation of $G$, by which we mean a discrete series representation of $SL(2,\R)$ that factors through $PSL(2,\R)$ --- so, $k$ must be even.  Then the conditions of Theorem \ref{ghjvndim} are satisfied, and we have a representation of $R\Gamma$ on $\HH$ of von Neumann dimension
\begin{align}
        \text{dim}_{R\Gamma} \HH &= 2\pi \left( 2g-2+\sum_{j=1}^l \left( 1-\frac{1}{m_j} \right) + h \right) \cdot \frac{k-1}{4\pi} \qquad
        (k=2,4,6,\ldots)
\end{align}
where we have used formulas (\ref{gb}) and (\ref{rformdim}), with $g, m_1, \ldots , m_l, h$ as in (\ref{Fgens}) and (\ref{Frelns}).
\end{examp}

\begin{rem} \label{projrep}
Suppose moreover that $\Gamma$ has trivial second cohomology group.  Then there is no need to restrict $k$ to be even: If $\pi$ is a projective representation of $PSL(2,\R)$ coming from a representation of $SL(2,\R)$, then the 2-cocycle associated to the restriction of the projective representation $\pi$ to $\Gamma$ is trivial, so we have a true representation of $\Gamma$ on $\HH$, and this representation extends to a representation of the II$_1$ factor $R\Gamma$.  
\end{rem}

\begin{examp} \label{ghjex} 
(\cite{ghj} Example 3.3.4)
Suppose $\Gamma$ is the free product of cyclic groups $\Z_2 * \Z_q$, where $q \geq 3$, which has  $g=0$, $m_1=2$, $m_2=q$, $h=1$ in \ref{Fgens} and \ref{Frelns}.  Then by Remark \ref{projrep}, we may start with representations of $SL(2,\R)$, even though $\Gamma$ is in $PSL(2,\R)$, and we have 
\begin{align*}
        \text{dim}_{R\Gamma} \HH &= \left( 1 - \frac{2}{q} \right) \cdot \frac{k-1}{4} \qquad
        (k=2,3,4,\ldots).
\end{align*}
\end{examp}

\begin{examp} \label{free}
Suppose in Example \ref{keven}, we take $\Gamma$ to be a free group $F_n$, which has $g=0$, no $m_i$, $h=n+1$ in \ref{Fgens} and \ref{Frelns}.  Then by Remark \ref{projrep}, we may start with representations of $SL(2,\R)$, even though $\Gamma$ is in $PSL(2,\R)$, and we have 
\begin{align*}
       \text{dim}_{RF_n} \HH &= (n-1) \cdot \frac{k-1}{2} \qquad
        (k=2,3,4,\dots).
\end{align*}
\end{examp}

\subsection{Corollary to a proof of Theorem \ref{ghjvndim}}

The main ingredients of the proof in \cite{ghj} of Theorem \ref{ghjvndim} are a locally compact unimodular group $G$ and a lattice $\Gamma$ in $G$ satisfying the following criteria:
\begin{itemize}
    \item[(i)] $\Gamma$ has the property that every one of its non-trivial conjugacy classes has infinitely many elements, and 
    \item[(ii)] $G$ has an irreducible unitary representation that is equivalent to a subrepresentation of the right regular representation.
\end{itemize}
When combined with Remark \ref{projrep} on projective representations, the proof of Theorem \ref{ghjvndim} gives Corollary \ref{padicvndim}: 

\begin{cor} \label{padicvndim} 
Let $F$ be a nonarchimedean local field of characteristic $0$, let $\Gamma$ be a lattice in $PGL(2,F)$, and let $\pi : GL(2,F) \rightarrow \UU(\HH)$ be a discrete series representation. Assume that every non-trivial conjugacy class of $\Gamma$ has infinitely many elements.  Then provided that $\Gamma$ has trivial second cohomology group, the restriction to $\Gamma$ of the projective representation of $PGL(2,F)$ obtained from the representation $\pi$ of $GL(2,F)$ gives a representation of $\Gamma$ on $\HH$ that extends to a representation of the \emph{II}$_1$ factor $R\Gamma$, with von Neumann dimension given by 
\begin{align*}
    \emph{dim}_{R\Gamma} \HH = \emph{covol}(\Gamma) \cdot \emph{d}_\pi,
\end{align*}
where $\emph{d}_\pi$ denotes the formal dimension of $\pi$.
\end{cor}

\section{Computations} \label{comp}

\subsection*{Notation}

If $F$ is a nonarchimedean local field, let $\fo_F$ denote its integers, let $\fp_F$ denote the maximal (prime) ideal of $\fo_F$, let $\varpi_F$ denote a prime element, so that $\varpi_F \fo_F =\fp_F $, let $q$ denote the order of the residue field, so that $\fo_F / \fp_F \cong \F_q $, and define the groups
\begin{align*}
    U_F^0 &= U_F = \fo_F^\times \\
    U_F^i &= 1 + \fp_F^i \qquad (i \geq 1).
\end{align*}
\subsection{Computing covolumes of torsion-free lattices in $PGL(2,F)$}

\begin{thm} \label{ihathm}
(\cite{iha} Theorem 1 and Corollary)
Let $F$ be a local nonarchimedean field.  Any torsion-free discrete subgroup $\Gamma$ is isomorphic to a free group on at most countably many generators.  If moreover $\Gamma \backslash PGL(2,F)$ is compact, then the number of free generators of $\Gamma$ is given by $$n=\frac{1}{2}(q-1)c+1,$$ 
where 
$$ c=\vert \Gamma \backslash PGL(2,F) / PGL(2,\fo_F) \vert .$$
\end{thm}

\begin{lem} \label{covol}
Let $\Gamma$ be as in Theorem \ref{ihathm}, and let $n$ be the number of it generators.  If Haar measure on $PGL(2,F)$ is normalized so that  
\begin{align*}
    \emph{vol}( PGL(2,\fo_F) )=\frac{1}{2}(q-1),  
\end{align*}
then
\begin{align*}
    \emph{vol}( \Gamma \backslash PGL(2,F) ) = n-1.
\end{align*}
\end{lem}

\begin{proof}
If Haar measure on $PGL(2,F)$ is normalized so that
\begin{align*}
    \text{vol}( PGL(2,\fo_F) )=1,
\end{align*}
then 
\begin{align*}
    \text{vol}( \Gamma \backslash PGL(2,F) ) = 
    \text{vol}( \Gamma \backslash PGL(2,F) / PGL(2,\fo_F) ),
\end{align*}
which by Theorem \ref{ihathm} is equal to 
\begin{align*}
    c = 2\frac{n-1}{q-1}.
\end{align*}
Covolume scales proportionally with Haar measure, so if Haar measure on $PGL(2,F)$ is normalized so that  
\begin{align*}
    \text{vol}( PGL(2,\fo_F) )=\frac{1}{2}(q-1),  
\end{align*}
then
\begin{align*}
    \text{vol}( \Gamma \backslash PGL(2,F) ) =
    \frac{1}{2}(q-1) \cdot 2\frac{n-1}{q-1} = n-1.
\end{align*}
\end{proof}

\subsection{Computing formal dimensions of discrete series representations of $GL(2,F)$}

The calculation of formal dimensions of discrete series representations of $GL(2,F)$ is already contained in the paper of Corwin, Moy, and Sally \cite{cms}, in which the authors calculate formal dimensions of discrete series representations of $GL(n,F)$, where $n$ is relatively prime to $q$, using Howe's notion of an admissible character of an extension of degree $n$ over $F$.  Here, we will carry out their same calculation in the notation of Bushnell and Henniart's book on $GL(2,F)$ \cite{buhe}, for the benefit of readers who may be unfamiliar with representation theory of $\fp$-adic groups, in case they wish to consult the book \cite{buhe} for more details.

Note that there are no Hilbert spaces in \cite{buhe}. But by Section 2 of \cite{cartier}, the irreducible smooth representations of $GL(2,F)$ in \cite{buhe} admit completions to the irreducible unitary representations of $GL(2,F)$ that we need for Corollary \ref{padicvndim}. 

Let $G=GL(2,F)=\text{Aut}_F(F \oplus F)$ and let $A=M_2(F)=\text{End}_F(F \otimes F)$. Let $Z$ denote the center of $G$, which we may identify with $F^\times$.

Let $\UU$ be a chain order in $A$.  Then either 
\begin{align*}
\UU = \MM = \begin{pmatrix}
\fo_F & \fo_F \\
\fo_F & \fo_F
\end{pmatrix},    
\end{align*} or
\begin{align*}
\UU = \JJ = \begin{pmatrix}
\fo_F & \fo_F \\
\fp_F & \fo_F
\end{pmatrix}.    
\end{align*}

Let $\fP$ denote the Jacobson radical of $\UU$, $\fP=\text{rad}\UU$.  There exists an element $\Pi \in G$ such that $\fP=\Pi \UU = \UU \Pi$.  So either 
\begin{align*}
\fP = \text{rad}\MM = \begin{pmatrix}
\varpi_F & 0 \\
0 & \varpi_F
\end{pmatrix} \MM,    
\end{align*} or
\begin{align*}
\fP = \text{rad}\JJ = \begin{pmatrix}
0 & 1 \\
\varpi_F & 0
\end{pmatrix}\JJ.    
\end{align*}

We define the groups 
\begin{align*}
    U_\UU^0 &= U_\UU = \UU^\times \\
    U_\UU^i &= 1 + \fP^i \qquad (i \geq 1).
\end{align*}

For example, 
\begin{align*}
    U_\MM & = GL(2,\fo_F) = K, \qquad & \text{a maximal compact subgroup of $G$} \\
    U_\JJ & = I, \qquad & \text{the Iwahori subgroup of $G$.}
\end{align*}

The surjection $\fo_F \rightarrow \fo_F / \fp_F \cong \F_q$ induces the maps
\begin{align} \label{tower}
    U_\MM \qquad & \longrightarrow \qquad GL(2, \F_q) \\
    U_\JJ \qquad & \longrightarrow \qquad 
    \begin{pmatrix}
    * & *  \\
    0 & *
    \end{pmatrix} \nonumber  \\
    U^1_\JJ \qquad & \longrightarrow \qquad 
    \begin{pmatrix}
    1 & * \nonumber \\
    0 & 1
    \end{pmatrix} \nonumber \\
    U^1_\MM \qquad & \longrightarrow \qquad 
    \begin{pmatrix}
    1 & 0 \\
    0 & 1
    \end{pmatrix}. \nonumber  
\end{align}

The normalizer of $U_\UU$ in $G$ is 
\begin{align*}
    N_G(U_\UU)=\langle \Pi \rangle \ltimes U_\UU.
\end{align*}

We now describe the discrete series representations of $G$: the Steinberg representation (and its twists by characters), and the supercuspidal representations. 

The \textit{Steinberg representation} of $G$ is induced from the trivial character on the subgroup of upper-triangular matrices in $G$.  This representation is shown to be square-integrable modulo $Z$ in 17.5-17.9 of \cite{buhe}.  Twists of the Steinberg representation by characters of $F^\times$ are also square-integrable modulo $Z$.    

\begin{lem} \label{Stdim}
If Haar measure on $G/Z$ is normalized so that  
\begin{align*}
    \emph{vol}( Z.K/Z )=\frac{1}{2}(q-1),  
\end{align*}
then the formal dimension of the Steinberg representation is 
\begin{align*}
    \emph{d}_{\emph{St}}=1.
\end{align*}
\end{lem}

\begin{proof}
We start from the last line of the proof that the Steinberg representation is square-integrable modulo $Z$ in \cite{buhe}, 17.5 Theorem, 
where Haar measure on $G/Z$ is normalized so that vol$(Z.I/Z)=1$.  
We need the group 
\begin{align*}
    \W_0 = \left\langle 
    \begin{pmatrix}
    0 & 1 \\
    1 & 0
    \end{pmatrix}, 
    \begin{pmatrix}
    0 & \varpi_F^{-1} \\
    \varpi_F & 0
    \end{pmatrix}
    \right\rangle.
\end{align*}
(This group appears because a set of coset representatives for $I \backslash G / I$ is given by the group $\W=\langle \Pi \rangle \ltimes \W_0$, where $\Pi = 
\begin{psmallmatrix}
0 & 1 \\
\varpi_F & 0
\end{psmallmatrix}
$; see \cite{buhe} 17.1 Theorem.)
Expanding a geometric series, using the fact that lengths of elements of $\mathbb{W}_0$ are $0,1,1,2,2,3,3,\ldots$, we have
\begin{align*}
\int_{G/Z} \vert f(g) \vert ^2 d \dot{g} &= 2 \sum_{g \in \mathbb{W}_0 } q^{-l(g)} 
= 2 \left( 2 \left( \frac{1}{1-\frac{1}{q}} \right) -1 \right) \\
&= 2 \left( 2 \left( \frac{q}{q-1} \right) -1 \right) 
= 2 \left( \frac{2q}{q-1} - \frac{q-1}{q-1} \right) \\
&= 2 \left( \frac{q+1}{q-1} \right).
\end{align*}
Because the matrix coefficient $f(g)$ is equal to $1$ when $g=\id$, 
we see that if Haar measure on $G/Z$ is normalized so that  
\begin{align*}
    \text{vol}( Z.I/Z )=1,  
\end{align*}
then the formal dimension of the Steinberg representation is 
\begin{align*}
    \text{d}_{\text{St}}=\frac{1}{2}(q-1)(q+1)^{-1}.
\end{align*}
Counting elements of the finite groups in the tower of surjections (\ref{tower}) gives 
$$\left[ Z.K/Z : Z.I/Z \right] = \left[ K : I \right] = q+1.$$
Formal dimension is inversely proportional to Haar measure, so if Haar measure on $G/Z$ is normalized such that  
\begin{align*}
    \text{vol}( Z.I/Z )=\frac{1}{2}(q-1)(q+1)^{-1},  
\end{align*}
which is equivalent to normalizing Haar measure on $G/Z$ such that 
\begin{align*}
    \text{vol}( Z.K/Z )=\frac{1}{2}(q-1),  
\end{align*}
then the formal dimension of the Steinberg representation is 
\begin{align*}
    \text{d}_{\text{St}}=\frac{1}{2}(q-1)(q+1)^{-1} \cdot \left( \frac{1}{2}(q-1)(q+1)^{-1} \right) ^{-1} = 1.
\end{align*}
\end{proof}

\begin{rem}
Lemma \ref{Stdim} above agrees with equation (2.2.2) in \cite{cms}, which says 
$$ \text{d}_{\text{St}} \cdot \text{vol}(K.Z/Z) = \frac{1}{n} \prod_{k=1}^{n-1}(q^k-1),$$
where $\text{St}$ denotes the Steinberg representation of $GL(n,F)$, and $K$ is a maximal compact subgroup of $GL(n,F)$.
\end{rem}

A \textit{supercuspidal representation} of $G$ is an irreducible unitary representation of $G$ that is not equivalent to a subrepresentation of any representation induced from a character on the subgroup of upper-triangular matrices obtained by inflating a character on the subgroup of diagonal matrices.  These are shown to have matrix coefficients that are compactly supported (hence square-integrable) modulo $Z$ in 10.1-10.2 in \cite{buhe}.  All supercuspidal representations of $G$ are ``compactly induced'' from irreducible finite-dimensional representations of certain open subgroups of $G$ that contain $Z$ and are compact modulo $Z$ (\cite{buhe} 20.2 Theorem). The matrix coefficients of the compactly induced representations are essentially those of the inducing representations (an observation of Mautner in \cite{mau}), which leads to Proposition \ref{cpctdim} below.

\begin{prop} \label{cpctdim} 
(\cite{knra} Proposition 1.2)  Let $\pi$ be a supercuspidal representation of $G$ compactly induced from an irreducible representation of $J<G$, where $Z<J$ and $J/Z$ is compact in $G/Z$.Then the formal dimension $\emph{d}_\pi$ of $\pi$ is given by
\begin{align*}
    \emph{d}_\pi = \frac{\emph{dim}\Lambda}{\emph{vol}(Z.J/Z)}.
\end{align*}
\end{prop} 

We start by describing the simplest supercuspidal representation of $G$ and calculating its formal dimension.  

\begin{examp} \label{dzsc} A \textit{cuspidal representation} of $GL(2,\F_q)$ is an irreducible  representation of $GL(2,\F_q)$ on a (finite-dimensional) complex vector space that is not equivalent to a subrepresentation of any representation induced from a character on upper-triangular matrices obtained by inflating a character on diagonal matrices; all these representations have dimension $q-1$ (\cite{ps} Proposition 10.2, or \cite{bum} Proposition 4.1.5).  Let $\Lambda_0$ be a cuspidal representation of $GL(2,\F_q)$.  Inflate this representation $\Lambda_0$ of $GL(2,\F_q) = U_\MM / U_\MM^1$ to a representation $\Lambda$ of $K=U_\MM$ trivial on $U^1_\MM$, then compactly induce $\Lambda$ from $K$ to a representation $\pi$ of $G$. This $\pi$ is a supercuspidal representation of $G$ (\cite{buhe} 11.5 Theorem (1), or \cite{bum} Theorem 4.8.1) .  
\end{examp}

\begin{lem} \label{dimdzsc}
If Haar measure on $G/Z$ is normalized so that  
\begin{align*}
    \emph{vol}( Z.K/Z )=\frac{1}{2}(q-1),  
\end{align*}
then the formal dimension of the supercuspidal representation in Example \ref{dzsc} is 
\begin{align*}
    \emph{d}_{\pi}=2.
\end{align*}
\end{lem}
\begin{proof}
From Proposition \ref{cpctdim},
\begin{align*}
    \text{d}_\pi = \frac{\text{dim}\rho}{\text{vol}(Z.J/Z)}
    = \frac{q-1}{\frac{1}{2}(q-1)} = 2.
\end{align*}
\end{proof}

We now describe the other supercuspidal representations and show how to compute their formal dimensions.  

Let $J$ be a compact open subgroup of $G$, and let $\Lambda$ be the finite-dimensional representation from which the supercuspidal representation $\pi$ of $G$ will be induced.  From Sections 15-16 of \cite{buhe}, $J$ is of the form
\begin{align*}
    J = E^\times U_\UU^{ \left\lfloor \frac{n+1}{2} \right\rfloor }
\end{align*}
where $E$ is an extension of $F$ of degree $2$; $\UU=\JJ$ if $E/F$ is ramified, and $\UU=\MM$ if $E/F$ is unramified; and $n\geq1$.  Note $J$ is contained in $\langle \Pi \rangle \ltimes U_\UU$, and $J$ contains and is compact modulo $Z$.  Each supercuspidal representation of $G$ besides the supercuspidal representation constructed in Example \ref{dzsc} corresponds to one of the following pairs of $J$ and $\Lambda$:
\[
\begin{tabular}{c|c|c|c} \label{scparam}
    $\UU$ & $E/F$ & $n$ &$\text{dim}\Lambda$ \\ \hline
    $\MM$ &  $\text{unramified}$ & $\text{even}$ & $q$ \\
    $\MM$ &  $\text{unramified}$ & $\text{odd}$ & $1$ \\
    $\JJ $ &  $\text{ramified}$ & $\text{odd}$ & $1$ 
\end{tabular}
\]

We now calculate the formal dimensions of these supercuspidal representations.  This amounts to calculating vol$(J)$, which will be calculated from the index of $J$ in $\langle \Pi \rangle \ltimes U_\UU$.  

\begin{rem}
We would like to emphasize that the calculations in Lemma \ref{index1} and Lemma \ref{index2} below are already contained in \cite{cms}.  All we are doing here is specializing to the case $n=2$ and using the notation of \cite{buhe}.
\end{rem}

\begin{lem} \label{index1}
Let $i\geq1$ be an integer.  We have 
\begin{itemize}
 \item[(i)]   $\left[ U_\UU : U_\UU^1 \right] = \vert GL(f, \F_q) \vert ^e$ 
 \item[(ii)]   $\left[ U^1_\UU : U^i_\UU \right] = q^{2f(i-1)}$ 
  \item[(iii)]  $\left[ U_E : U^1_E \right] = (q^f-1)$ 
  \item[(iv)]  $\left[ U^1_E : U^i_E \right] = (q^f)^{i-1}$
  \item[(v)] $\left[ \langle \Pi \rangle \ltimes U_\UU : U_\UU \right] = e$
\end{itemize}
where $e=1$ and $f=2$ if $\UU=\MM$ and $E/F$ is unramified; and $e=2$ and $f=1$ if $\UU=\JJ$ and $E/F$ is ramified. 
\end{lem}
\begin{proof}
For (i), count the elements of the finite groups in the tower of surjections (\ref{tower}).  

For (ii), use 
\begin{align*}
    [1+\fP : 1 + \fP^i]=[\fP : \fP^1].
\end{align*}
For (iii), note 
\begin{align*}
    \vert U_E / U_E^1 \vert = \vert \F_{q^f} ^\times \vert.
\end{align*}
For (iv), use 
\begin{align*}
    [1+\fp_E : 1 + \fp_E^i]=[\fp_E : \fp_E^1].
\end{align*}
For (v), note 
\begin{align*}
\left[ \langle \Pi \rangle \ltimes U_\UU : U_\UU \right] = \vert \langle \Pi \rangle / Z \vert. 
\end{align*}
\end{proof}

\begin{lem} \label{index2}
$$\left[ \langle \Pi \rangle \ltimes U_\UU : E^\times U^i_\UU \right] = \frac{\vert GL(f, \F_q) \vert^e  q^{2f(i-1)}}{(q^f-1)(q^f)^{i-1}}$$
where $e=1$ and $f=2$ if $\UU=\MM$ and $E/F$ is unramified; and $e=2$ and $f=1$ if $\UU=\JJ$ and $E/F$ is ramified. 
\end{lem}
\begin{proof}
We have
\begin{align*}
\left[ \langle \Pi \rangle \ltimes U_\UU : E^\times U^i_\UU \right] &= \left[ U_\UU : U_E U^i_\UU \right] \\
&= \frac{ \left[ U_\UU : U_\UU^i \right] }{\left[ U_E U_\UU^i : U_\UU^i \right]} \\ 
&= \frac{ \left[ U_\UU : U_\UU^1 \right] \left[ U^1_\UU : U_\UU^i \right] }{\left[ U_E : U_E^1 \right] \left[ U^1_E : U_E^i \right]},
\end{align*}
and plugging in Lemma \ref{index1} (i)-(iv) gives the result.   
\end{proof}

\begin{lem} \label{scformdim}
If Haar measure on $G/Z$ is normalized so that  
\begin{align*}
    \emph{vol}( Z.K/Z )=\frac{1}{2}(q-1),  
\end{align*}
then the volumes of $J/Z$ and the formal dimensions of the supercuspidal representations compactly induced by each $\Lambda$ in the previous table are given by the additional columns: 
\[
\begin{tabular}{c|c|c|c|c|c} 
    $\UU$ & $E/F$ & $n$ &$\emph{dim}\Lambda$     & \emph{vol}$(J/Z)$ & \emph{d}$_\pi$
    \\ \hline
    $\MM$ &  $\emph{unramified}$ & $\emph{even}$ & $q$ & $\frac{1}{2}q^{-(2i-1)}$ &  $2q^{2i}$\\
    $\MM$ &  $\emph{unramified}$ & $\emph{odd}$ & $1$ & $\frac{1}{2}q^{-(2i-1)}$ & $2q^{2i-1}$\\
    $\JJ $ &  $\emph{ramified}$ & $\emph{odd}$ & $1$ & $\frac{1}{q+1}q^{-(i-1)}$ & $(q+1)q^{i-1}$
\end{tabular}
\]
where $i = \left\lfloor \frac{n+1}{2} \right\rfloor$.
\end{lem}
\begin{proof}
First, 
$$ \text{vol}(J/Z) = \frac{\text{vol}(\langle \Pi \rangle \ltimes U_\UU / Z)}{\left[ \langle \Pi \rangle \ltimes U_\UU / Z: E^\times U^i_\UU / Z\right]}.$$  

\textbf{Case 1}:  $E/F$ is unramified, so $\UU=\MM$, $e=1$, and $f=2$.  From Lemma \ref{index1} (v), 
$$\text{vol}(\langle \Pi \rangle \ltimes U_\UU / Z) = \text{vol}(Z.K/Z) = \frac{1}{2}(q-1),$$ 
and from Lemma \ref{index2}, 
$$ \left[ \langle \Pi \rangle \ltimes U_\UU / Z: E^\times U^i_\UU / Z\right] = 
(q^2-q)q^{2(i-1)},$$
so 
$$\text{vol}(J/Z) = \frac{1}{2}q^{-(2i-2)-1}.$$

\textbf{Case 2}:  $E/F$ is ramified, so $\UU=\JJ$, $e=2$, and $f=1$.  From Lemma \ref{index1} (v) and the tower of surjections (\ref{tower}),
$$\text{vol}(\langle \Pi \rangle \ltimes U_\UU / Z) = 2 \text{vol}(Z.I/Z) = 2 \frac{1}{q+1} \text{vol}(Z.K/Z) = \frac{q-1}{q+1},$$ 
and from Lemma \ref{index2}, 
$$ \left[ \langle \Pi \rangle \ltimes U_\UU / Z: E^\times U^i_\UU / Z\right] = 
(q-1)q^{i-1},$$
so
$$\text{vol}(J/Z) =  \frac{1}{q+1}q^{-(i-1)}.$$

Dividing dim$\Lambda$ by these volumes according to Proposition \ref{cpctdim} completes the table.
\end{proof}

\begin{rem}
These formal dimensions agree with the formal dimensions in Theorem 2.2.8 of \cite{cms} for $n=2$.
\end{rem}

\subsection{Proof of Example \ref{padicfree} }

Using the formula in Corollary \ref{padicvndim}, we may multiply the covolumes in Lemma \ref{covol} by the formal dimensions in Lemma \ref{Stdim}, Lemma \ref{dimdzsc}, Lemma \ref{scformdim} to obtain the list of von Neumann dimensions in Example \ref{padicfree}.

\section{Concluding remarks}

\subsection{Related directions}  

In \cite{ruthdiss}, formulas for dimensions of spaces of cusp forms were combined with Theorem \ref{ghjvndim} to obtain representations of the II$_1$ factor $R\Gamma$ on subspaces of $L^2(\Lambda \backslash G)$, where $G=PSL(2,\R)$, and $\Gamma$ and $\Lambda$ are lattices in $G$. In \cite{ruthlmp}, we will compile a list of situations where the pointwise limit multiplicity has been proven to hold, thereby obtaining representations of a II$_1$ factor $R\Gamma$ on subspaces of $L^2(\Lambda \backslash G)$ for many other groups $G$, including $G=PSL(2,\R) \times PSL(2,\R)$ and $G=PGL(3,F)$.  

\subsection{Further questions}

\begin{que}
Ihara's Theorem \ref{ihathm} only concerns torsion-free lattices in $PGL(2,F)$.  What about lattices in $PGL(n,F)$, $n \geq 3$?  If we knew the covolumes of these lattices, we could extend Example \ref{padicfree} using the rest of the formal dimensions in \cite{cms}.  (We are not asking about lattice covolume in $PSL(n,\R)$, $n \geq 3$, because $SL(n,\R)$ has discrete series representations only for $n=2$.)
\end{que}

\begin{que}
$PSL(2,\R)$ contains lattices that, unlike free groups, do not have trivial second cohomology group, \textit{e.g}.\ lattices isomorphic to $\pi_1(X)$, where $X$ is a compact Riemann surface.  Could it be that the restriction to $\pi_1(X)$ of a 2-cocycle associated to a projective representation of $PSL(2,\R)$ \textit{arising from a representation of $SL(2,\R)$} is trivial anyway?  What is the relationship between central extensions of an algebraic group on the one hand, and central extensions of its lattices on the other?
\end{que}

\begin{que}
Is it possible to construct representations of II$_1$ factors using principal series representations, rather than  discrete series representations?  The proof of Theorem \ref{ghjvndim} relies crucially on the fact that discrete series representations occur discretely in the right regular representation, so the method does not carry over to principal series representations, which occur continuously.  
\end{que}

\begin{que}
The index of a subfactor is defined as the ratio of two von Neumann dimensions: Data about the ``size'' of the representation space cancels out, leaving data about the subfactor inclusion.  
Vaughan Jones showed in \cite{jon} that the index of a subfactor in a II$_1$ factor can only assume values in the set 
$$ \lbrace 4 \cos ^2 \left( \pi / q \right) \, \mid \, q=3,4,5,\ldots \rbrace \cup [4,\infty],$$ and he gave an example of a factor in which all these subfactor indices are achieved.  For a given lattice $\Gamma$ in $PSL(n,\R)$ or $PGL(n,F)$, what subset of these index values is achieved in $R\Gamma$?  Does the answer depend upon whether $n=2$ or $n>2$?
\end{que}

\section*{Acknowledgments} 

We thank Alain Valette, A. Raghuram, Moshe Adrian, Allen Moy, and Vaughan Jones for enlightening conversations and correspondence.  We are especially grateful to Vaughan Jones for Remark \ref{projrep}.  This work would not have been possible without the workshop and conference on representation theory of $\fp$-adic groups at IISER Pune during July 2017 and the Hausdorff Trimester Program on von Neumann algebras in Bonn during Summer 2016; we thank the organizers for the opportunity to have attended these events. 

\bibliographystyle{amsalpha}
\bibliography{ref}

\end{document}